\documentclass[11pt]{amsart}

\usepackage{fourier}
\usepackage{amsmath}%
\usepackage{amsfonts}%
\usepackage{amssymb}%
\usepackage{color}
\usepackage{graphicx}
\usepackage{setspace}
\usepackage{enumerate}
\usepackage{mathrsfs}
\usepackage{amsrefs}
\usepackage[colorlinks = true, citecolor = blue]{hyperref}
\usepackage{float}
\usepackage{comment}

\allowdisplaybreaks

\title[Singular integrals on $C^{1,\alpha}$ regular curves in $\mathbb{G}$]{Singular integrals on $C^{1,\alpha}$ regular curves in Carnot groups}
\author{Vasileios Chousionis, Sean Li, and Scott Zimmerman}

\address{V.\ Chousionis: Department of Mathematics, University of Connecticut, 
	341 Mansfield Road U1009, Storrs, Connecticut 06269, USA, {\tt vasileios.chousionis@uconn.edu}}

\address{S.\ Li: Department of Mathematics, University of Connecticut, 
	341 Mansfield Road U1009, Storrs, Connecticut 06269, USA, {\tt sean.li@uconn.edu}}
	
\address{S.\ Zimmerman: Department of Mathematics, University of Connecticut, 
	341 Mansfield Road U1009, Storrs, Connecticut 06269, USA, {\tt scott.zimmerman@uconn.edu}}

\thanks{V.C. is supported by  the Simons Collaboration grant no.\  521845.}

plus 6pt minus 12pt
\makeatletter
\newtheorem*{rep@theorem}{\rep@title}
\newcommand{\newreptheorem}[2]{%
\newenvironment{rep#1}[1]{%
 \def\rep@title{#2 \ref{##1}}%
 \begin{rep@theorem}}%
 {\end{rep@theorem}}}
\makeatother

\newtheorem{theorem}{Theorem}
\newreptheorem{theorem}{Theorem}
\newtheorem{lemma}[theorem]{Lemma}

\newtheorem{proposition}[theorem]{Proposition}

\def\diam{{\rm diam\,}}
\def\bo{{\textbf{0}\,}}
\def\dist{{\rm dist\,}}

\def\supp{{\rm supp\,}}

\def\reg{{\rm reg \,}}

\theoremstyle{definition}
\newtheorem{remark}[theorem]{Remark}
\newtheorem{definition}[theorem]{Definition}
\newtheorem{example}[theorem]{Example}

\newcommand{\barint}{
\rule[.036in]{.12in}{.009in}\kern-.16in \displaystyle\int }

\newcommand{\barcal}{\mbox{$ \rule[.036in]{.11in}{.007in}\kern-.128in\int $}}

\newcommand{\G}{\mathbb G}
\newcommand{\ra}{\rightarrow}

\newcommand{\stm}{\setminus}
\newcommand{\ve}{\varepsilon}

\newcommand{\R}{\mathbb R}

\def\diam{\operatorname{diam}}

\def\dist{\operatorname{dist}}

\def\supp{\operatorname{supp}}

\def\mvint_#1{\mathchoice
          {\mathop{\vrule width 6pt height 3 pt depth -2.5pt
                  \kern -8pt \intop}\nolimits_{\kern -3pt #1}}%
          {\mathop{\vrule width 5pt height 3 pt depth -2.6pt
                  \kern -6pt \intop}\nolimits_{#1}}%
          {\mathop{\vrule width 5pt height 3 pt depth -2.6pt
                  \kern -6pt \intop}\nolimits_{#1}}%
          {\mathop{\vrule width 5pt height 3 pt depth -2.6pt
                  \kern -6pt \intop}\nolimits_{#1}}}

\numberwithin{theorem}{section} \numberwithin{equation}{section}

\begin{document}
\begin{abstract} Let $\G$ be any Carnot group. We prove that if a convolution type singular integral associated with a $1$-dimensional Calder\'on-Zygmund kernel is $L^2$-bounded on horizontal lines, with uniform bounds, then it is bounded in $L^p, p \in (1,\infty),$ on any compact $C^{1,\alpha}, \alpha \in (0,1],$ regular curve in $\G$. 
\end{abstract}

\maketitle

\section{Introduction}

In this paper we will study the boundedness of  convolution type Singular Integral Operators (SIOs) on smooth, $1$-dimensional subsets of arbitrary Carnot groups. We will consider SIOs formally given by
$$
T f(p) = \int K(q^{-1}p) f(q) \, d \mu(q),
$$
where $K$ is a $1$-dimensional Calder\'on-Zygmund (CZ) kernel and $\mu$ is the restriction of the Hausdorff 1-measure $\mathcal{H}^1$ to a smooth, regular curve in a Carnot group $\mathbb{G}$. In an effort to keep the introduction concise and rather informal, we will defer all definitions to Section~\ref{sec_prelim}.

The study of SIOs on lower dimensional subsets of Euclidean space has been a highlight of the interface between harmonic analysis and geometric measure theory, see e.g. \cite{Calderon,CMM, david-wavelets,DS2,MatMelVer,NazTolVolCodim}. Advances in the area have partly been motivated by the importance of lower dimensional SIOs in complex analysis, potential theory, and PDE. For example, certain singular integrals, such as the Cauchy and Riesz transforms, play crucial roles in the study of removability for bounded analytic or Lipschitz harmonic functions, see e.g. \cite{tolsabook,NazTolVolHarm}. 

Singular integrals in Carnot groups, defined with respect to the corresponding Haar measure, have been studied extensively since the early 70's, see e.g. \cite{stein,stfol}. However,  SIOs on lower dimensional subsets of sub-Riamannian spaces were first considered relatively recently in \cite{ChoMat} and further studied in \cite{ CFOsios,ChoLi, ChoLiZimTra, CMrem, FOflag, FOcurves}. Analogously to the Euclidean case, some of these investigations have been motivated by the emergence of lower dimensional singular integrals in the study of removability for Lipschitz harmonic functions  in Carnot groups \cite{CFOsios,CMT, CMrem}. In such a setting, harmonic functions are solutions to sub-Laplacian equations. 
For an extensive account on sub-Laplacians, see \cite{Italians}. 

Research programs into SIOs on lower dimensional subsets of Euclidean spaces exhibit a common characteristic: all of the studied kernels are odd. Indeed, for a SIO to make sense on lines and other ``nice'' $1$-dimensional Euclidean objects, one employs cancellation properties of the kernel, see e.g. \cite[Propositions 1 and 2, p. 289]{stein}. Surprisingly, the situation is very different in Carnot groups. More specifically, it was shown in \cite{ChoLi,ChoLiZimTra} that, in any Carnot group, there exist non-negative and symmetric kernels (which we will call \emph{vertical Riesz kernels}) that define $L^2$ bounded operators on all $1$-regular curves. See Example~\ref{vriesz} for more details. 

F\"assler and Orponen \cite{FOcurves}, in a significant recent contribution, proved that $C^\infty$, $-1$-homogeneous and horizontally antisymmetric kernels (see Remark \ref{sym+hsym}) define convolution type SIOs  which are bounded in $L^p, p \in (1,\infty),$  for any regular curve in the first Heisenberg group. Recall that the first Heisenberg group is the simplest non-abelian example of a Carnot group. Although this result applies to a very broad class of kernels, it does not apply to the vertical Riesz kernels mentioned above. Aiming for a framework that will also encompass these examples, F\"assler and Orponen ask in \cite[Question 1]{FOcurves} if any smooth, $-1$-homogeneous CZ kernel in the  Heisenberg  group which is also uniformly $L^2$ bounded on horizontal lines (see Definition \ref{ubhl}) is necessarily $L^2$ bounded on regular curves. Theorem~\ref{main} below gives a partial answer to their question if the curve is further assumed to be $C^{1,\alpha}$ regular. Moreover, our result holds in any arbitrary Carnot group, and we do not need to assume $-1$-homogeneity and $C^{\infty}$ smoothness for the kernels.


\begin{theorem}
\label{main}
Let $\alpha \in (0,1]$ and suppose that $\Gamma$ is a $C^{1,\alpha}$ regular curve in a Carnot group $\mathbb{G}$. Then any convolution type singular integral operator with a $1$-dimensional CZ kernel which is $L^2$ bounded on horizontal lines with uniform bounds, it is $L^p(\mathcal{H}^1|_E)$-bounded
for any 1-regular set $E \subset \Gamma$ and any $p \in (1,\infty)$.
\end{theorem}

Moreover, the conclusion of Theorem \ref{main} applies to any 1-regular measure $\mu$ whose support is contained in $\Gamma$. An interesting aspect of Theorem \ref{main} is that it reduces the problem of determining $L^p (\mathcal{H}^1|_\Gamma)$-boundedness of a Carnot convolution type SIO to a problem of Euclidean $1$-dimensional SIOs. Indeed, the restriction of a Carnot $1$-dimensional CZ kernel (as in Definition~\ref{carnotcz}) to a horizontal line essentially induces a Euclidean SIO of convolution type associated with a standard $1$-dimensional CZ kernel. See the proof of Proposition~\ref{ubhlab} for more details. This phenomenon was first observed in \cite{CFOsios}. However, the reduction in that paper was to $3$-dimensional Euclidean kernels acting on $2$-dimensional planes, while, in our case, the problem is reduced to the action of $1$-dimensional kernels on lines.

The proof of Theorem \ref{main} is performed using the ``good lambda method" and an application of the  $T1$ theorem in homogeneous spaces. In order to verify the $T1$ testing conditions (on the corresponding Christ cubes) we employ a Littlewood-Paley decomposition of the operator as in \cite{CFOsios, ChoLi, ChoLiZimTra}. We stress that a key ingredient in our proof involves horizontal approximation of smooth curves. It is known from Pansu's Theorem \cite{Pansu} that Lipschitz curves in Carnot groups are well approximated by their horizontal tangent lines. We will show in Proposition~\ref{flat} (which is inspired by Monti's work in \cite{MontiThesis}) that, if a Lipschitz curve is further assumed to be $C^{1,\alpha}$,  then this approximation by horizontal tangents is quantitatively strong. 
This was essential in our application of the $T1$ theorem.

The paper is organized as follows.
In Section~\ref{sec_prelim}, we discuss (and introduce basic properties) of Carnot groups, curves in Carnot groups, and singular integral operators. Moreover we include two examples of families of kernels which satisfy the assumptions of Theorem \ref{main}. In Section \ref{sec-tools} we include some essential tools for the proof of Theorem \ref{main}. In particular, we prove Proposition \ref{flat}. Section \ref{sec-main} is devoted to the proof of Theorem \ref{main}.  

Throughout the article, we will write $a \lesssim b$ if $a \leq Cb$ for some constant $C>0$ which 
depends only on the group structure of $\mathbb{G}$,
 the curve $\Gamma$, and the kernel $K$.
 
 \textbf{Acknowledgements.} We thank Katrin F\"assler and Tuomas Orponen for several helpful comments regarding the good $\lambda$ method.

\section{Preliminaries}
\label{sec_prelim}

\subsection{Carnot groups}

A connected, simply connected Lie group $\mathbb{G}$ is called a \emph{step-$s$ Carnot group}
if the associated Lie algebra $\mathfrak{g}$ is \emph{stratified} in the following sense:
$$
\mathfrak{g} = V_1 \oplus \cdots \oplus V_s, 
\quad [V_1,V_i] = V_{i+1}  \text{ for } i=1,\dots,s-1,
\quad [V_1,V_s] =\{0\}
$$
where $V_1,\dots,V_s$ are non-zero subspaces of the Lie algebra.
Choose a basis $\{X_1,\dots,X_N\}$ of $\mathfrak{g}$
adapted to the stratification in the following sense:
$$
\left\{X_{\sum_{j=1}^{i-1} (\dim V_j) +1},\dots,X_{\sum_{j=1}^{i} (\dim V_j)} \right\}
\text{ is a basis of } V_i
\text{ for each } i \in \{1,\dots,s\}.
$$
Since the first layer is arguably the most important layer in any Carnot group, for simplicity we will set $n := \dim V_1$. For any $x \in \mathbb{G}$,
we can uniquely write $x = \text{exp}(x_1X_1 + \cdots + x_N X_N)$ for some $(x_1,\dots,x_N) \in \mathbb{R}^N$ via the exponential map $\text{exp}:\mathfrak{g} \to \mathbb{G}$.
Thus we may identify $\mathbb{G}$ with $\mathbb{R}^N$ using the relationship 
$x \leftrightarrow (x_1,\dots,x_N)$.
Note that, under this identification, we have $p^{-1} = -p$ for any $p \in \mathbb{G}$.
Denote by $|\cdot|$ the Euclidean norm in $\mathbb{G} = \mathbb{R}^N$ (depending on the above choice of basis).
For a general discussion of Carnot groups, see \cite{Italians}.

There is a natural family of automorphisms known as \emph{dilations} on $\mathbb{G}$.
Each coordinate $j$ satisfies
$$
\dim V_1 + \cdots + \dim V_{d_j - 1} < j \leq \dim V_1 + \cdots + \dim V_{d_j}
$$
for some $d_j \leq s$, and we call $d_j$ the \emph{degree} of the coordinate $j$.
We thus define for any $t>0$ the {\em dilation}
$$
\delta_t(x) = \left( t x_1,t^{d_2}x_2, \dots, t^{d_N}x_N \right).
$$
It follows that $\{\delta_t\}_{t>0}$ is a one parameter family of automorphisms i.e. $\delta_u \circ \delta_t = \delta_{ut}$. 
Note that if  $p = (p_1,\dots,p_s) \in \mathbb{G}$ where
$p_i \in \mathbb{R}^{\dim V_i}$,
then for any $t>0$ 
$$
\delta_t(p) = \left( tp_1,t^{2}p_2, \dots,t^{s}p_s \right).
$$
Moreover, we define the \emph{horizontal projection} $\tilde{\pi}:\mathbb{G} \to \mathbb{G}$ by $\tilde{\pi}(p_1, \dots, p_s) = (p_1,0, \dots,0)$ where $p_i \in \mathbb{R}^{\dim V_i}$ for $i=1,\dots,s$. In some instances we will denote $q \in \G$ as $(q_1,q_2)$ where $q_1 \in \mathbb{R}^{n}$ and $q_2 \in \mathbb{R}^{N-n}$, in particular, we will write $q=(q_1, \bo)$ for  $\bo=(0,\dots,0) \in \mathbb{R}^{N-n}$.

Since the Lie algebra is nilpotent, we may explicitly compute the group operation
using the famous Baker-Campbell-Hausdorff formula (see \cite{Dynkin}).
We will collect some useful properties of the group law here
(see \cite[ Lemma~1.7.2]{MontiThesis} or \cite[Proposition~2.2.22]{Italians}).
\begin{proposition}
\label{p:grouplaw}
We may write the group law as $xy = x + y + Q(x,y)$ for some polynomial $Q = (Q_1,\dots,Q_N)$ where
\begin{enumerate}
\item \label{groupl1} $Q_1 = \cdots = Q_n = 0$;
\item for $n < i \leq N$, the polynomial $Q_i(x,y)$ is a sum of terms each of which contains a factor of the form $(x_jy_\ell - x_\ell y_j)$ for some $1 \leq \ell,j < i$;
\item $Q_i$ is homogeneous of degree $d_i$ with respect to dilations (i.e. $t^{d_i} Q_i(x,y) = Q_i(\delta_t(x),\delta_t(y))$ for all $x,y\in\mathbb{G}$). \label{groupl3}
\item If the coordinate $x_i$ has degree $d_i \geq 2$, then $Q_i(x, y)$ depends only on the
coordinate of $x$ and $y$ which have degree strictly less than $d_i$. \label{groupl4}
\end{enumerate}
\end{proposition}

The are many choices of metric space structure for a Lie group $\mathbb{G}$.
However, any left-invariant, homogeneous metric $d$
(i.e. any metric which satisfies $d(px,py) = d(x,y)$ and $d(\delta_t(x),\delta_t(y)) = td(x,y)$)
is bi-Lipschitz equivalent to any other left-invariant, homogeneous metric $d'$ in the following sense: there is a constant $C \geq 1$ such that
$$
C^{-1}d'(x,y) \leq d(x,y) \leq C d'(x,y) \quad \text{ for any } x,y \in \mathbb{G}.
$$
The following implies that the topologies of $(\mathbb{R}^N,|\cdot|)$ and $(\mathbb{G},d)$ are equal
for any choice of left-invariant, homogeneous metric $d$. See, for example, Proposition~1.1.4 and Theorem~1.5.1 from \cite{MontiThesis}
or 
Proposition~5.1.6 from \cite{Italians}.
\begin{proposition}
\label{compact}
Suppose $\mathbb{G}$ has step $s$, and let $A \subset \mathbb{G}$ be compact.
Then there is a constant $D \geq 1$ such that
$$
D^{-1} |x-y| \leq d(x,y) \leq D|x-y|^{1/s}
\quad
\text{for all }
x,y \in A.
$$
\end{proposition}

One particularly useful left-invariant, homogeneous metric is defined as follows.
Define the norm $\Vert \cdot \Vert$ on $\mathbb{G}$ as
$$
\Vert x \Vert := \max_{j=1,\dots,N} \left\{ \lambda_j \left|x_j \right|^{1/{d_j}} \right\}
$$
where the constants $\lambda_j > 0$ are chosen (based on the group structure)
so that $\Vert \cdot \Vert$ satisfies the triangle inequality and $\lambda_1 = \cdots = \lambda_n = 1$.
(Such a choice can always be made; see \cite{Yves}.)
Define $d_\infty (x,y) = \Vert y^{-1}x \Vert$ for any $x,y \in \mathbb{G}$.
One may easily check that $d_\infty$ is indeed left-invariant and homogeneous.

\subsection{Curves in Carnot groups}
Suppose $\alpha \in (0,1]$.
A function $f:[a,b] \to \mathbb{R}$ is of class $C^{1,\alpha}$ 
if the derivative of $f$ exists and is $\alpha$-H\"{o}lder continuous on $[a,b]$.
(Differentiation at the endpoints is understood in terms of left and right hand limits.)
That is, for some $C_\alpha \geq 1$, 
$$
|f'(x)-f'(y)| \leq C_\alpha |x-y|^\alpha
\quad
\text{for all } x,y \in [a,b].
$$
A mapping $\gamma:[a,b] \to \mathbb{R}^N$ is of class $C^{1,\alpha}$
if each coordinate of $\gamma$ is of class $C^{1,\alpha}$.
We will say that a set $\Gamma$ is a curve if it is equal to the 
image of a Lipschitz map
$\gamma:[a,b] \to \mathbb{G}$,
and we say that $\gamma$ is $C^{1,\alpha}$ if it 
is $C^{1,\alpha}$ as a mapping into the ambient Euclidean space $\mathbb{R}^N$.

\begin{definition}
\label{regularmeas} We say that a Radon measure  $\mu$ on $\mathbb{G}$ is $1$-{\em regular} if
there exists some constant $C \geq 1$:
\begin{equation}
\label{meas-reg}
C^{-1} r \leq \mu(B(p,r)) \leq Cr \quad \text{ for any } p \in \mathbb{G}, \, 0 < r \leq \text{diam}(\supp \, \mu).
\end{equation}
We will denote by $\reg (\mu)$ the smallest constant $C \geq 1$ such that \eqref{meas-reg} holds.  If only the second inequality in \eqref{meas-reg} holds, $\mu$ is called  \emph{upper $1$-regular}. 

Moreover, a closed set $E \subset \G$  will be called \emph{$1$-regular}, if $\mathcal{H}^1|_{E}$ (the restriction of the $1$-dimensional Hausdorff measure on $E$) is $1$-regular. Analogously, we also define  \emph{upper $1$-regular} sets.
\end{definition}

\begin{definition}
A curve $\Gamma \subset \mathbb{G}$ is a \emph{$C^{1,\alpha}$ regular curve}
if it is a $1$-regular set
whose arc length parameterization (with respect to the metric on $\mathbb{G}$) is $C^{1,\alpha}$.
\end{definition}

\begin{remark}
The claim that the arc length parameterization 
of $\Gamma$ is $C^{1,\alpha}$ is more restrictive than necessary. 
It is an exercise in calculus to show that, 
if $\gamma:[a,b] \to \mathbb{G}$ is any $C^{1,\alpha}$, Lipschitz curve
and $|\gamma'| > 0$,
then the arc length parameterization of $\gamma$ must be $C^{1,\alpha}$ as well.
\end{remark}

\begin{remark}
Due to the stratified structure of the Carnot group, we do not need to assume that $\gamma$ is $C^{1,\alpha}$ in every coordinate.
Indeed, we need only to assume that the first $n$ coordinates (which are the coordinates in the first layer of $\mathbb{G}$) 
are $C^{1,\alpha}$, and smoothness of the remaining coordinates would follow.
This is a much more intrinsic assumption in the Carnot setting.
However, we will assume full regularity of the curve for simplicity.
\end{remark}

The following lemma is a fundamental fact in Carnot groups. 
It states that the Pansu derivative of a Lipschitz curve lies solely in the first layer almost everywhere.
By Proposition~\ref{compact}, every Lipschitz curve in $\mathbb{G}$ is also Lipschitz as a curve in $\mathbb{R}^N$.
In particular, such curves are classically differentiable almost everywhere.
\begin{lemma}
\label{montilem}
If $\gamma:[0,1] \to \mathbb{G}$ is Lipschitz
then 
$$
\lim_{s \searrow 0} \delta_{1/s}(\gamma(t)^{-1}\gamma(t+s)) = (\gamma_1'(t),\dots,\gamma_n'(t),0,\dots,0)
\quad
\text{ for a.e. } t \in [0,1].
$$
\end{lemma}
For a proof, see \cite[Lemma 2.1.4]{MontiThesis}.
Note that the cited lemma uses the terminology $h_i$ in place of $\gamma_i'$
where $h$ is the vector of canonical coordinates of the Lipschitz curve $\gamma$ with respect to the basis of $\mathfrak{g}$.
However, in a Carnot group, 
we may always choose a basis so that
$h_i = \gamma_i'$ a.e. for $i = 1, \dots, n$.
See, for example, Corollary~1.3.19 and Remark~1.4.5 in \cite{Italians}.

Lemma \ref{montilem} tells us that the tangents
to Lipschitz curves are 
(left translates of)
lines which vanish outside of the horizontal layer.
This inspires the following definition.
\begin{definition}
A set $L \subset \mathbb{G}$ is a {\em horizontal line} if
$$
L=\{ (sx_1,\dots,sx_n,0,\dots,0) \, : \, s \in \R \}.
$$
\end{definition}
Note that left translates of horizontal lines remain lines only in step 2 Carnot groups; in Carnot groups of higher step they are polynomial curves!

\subsection{Singular integral operators}

For the remainder of the paper, suppose that $d$ is any left-invariant, homogeneous metric on $\mathbb{G}$.

\begin{definition}
\label{carnotcz}
A continuous function $K: \mathbb{G} \setminus \{0\} \to \mathbb{R}$ is a \emph{1-dimensional Calder\'{o}n-Zygmund (CZ) kernel} if there exist constants $B>0$ and $ \beta \in (0,1]$, such that $K$ satisfies the \textit{growth condition}
\begin{equation}
\label{CZ1}
|K(p)| \leq  \frac{B}{d(p, 0)}
\end{equation}
and the \textit{H\"{o}lder continuity condition}
\begin{equation}
\label{CZ2}
|K(q^{-1}p_1) - K(q^{-1}p_2)|
+
|K(p_1^{-1}q) - K(p_2^{-1}q)|
\leq B
\frac{d(p_1, p_2)^{\beta}}{d(p_1,q)^{1+\beta}}\end{equation}
for  any $p \in \mathbb{G} \setminus \{0\}$, and any $p_1,p_2,q \in \mathbb{G}$ with $d(p_1,p_2) \leq d(p_1,q) / 2$.
\end{definition}


Fix a 1-dimensional CZ kernel $K: \mathbb{G} \setminus \{0\} \to \mathbb{R}$ and an upper $1$-regular measure $\mu$ (recall Definition \ref{regularmeas}).
Define for any $\varepsilon > 0$ the truncated SIO $T_{\mu,\varepsilon}$ associated with $K$ as
$$
T_{\mu,\varepsilon} f(p) = \int_{d( p,q) > \varepsilon} K(q^{-1}p) f(q) \, d \mu(q)
\quad
\text{ for any }
f \in L^p(\mu), \; 1 < p < \infty,
$$	
and define the maximal SIO $T_*$ associated with $K$ as
\begin{equation}
\label{maxSIO}
T_{\mu,*} f(p) = \sup_{\varepsilon > 0} \left| T_{\mu,\varepsilon} f(p) \right|
\quad
\text{ for any }
f \in L^p(\mu), \; 1 < p < \infty.
\end{equation}

\begin{definition}
If $K: \mathbb{G} \setminus \{0\} \to \mathbb{R}$
is a 1-dimensional CZ kernel, $1 < p < \infty$,
and $\mu$ is an upper $1$-regular measure, 
we say that {\em the singular integral operator $T_{\mu}$ associated with $K$  is bounded on $L^p(\mu)$} 
if the operators $f \mapsto T_{\mu,\varepsilon}f$ are bounded on $L^p(\mu)$ uniformly for all $\varepsilon >0$ 
(i.e. the constants are independent of the choice of $\varepsilon$). We also denote
$$\|T_{\mu}\|_{L^p(\mu) \ra L^p(\mu)}=\sup\{C>0: \|T_{\mu,\ve}f\|_{L^p(\mu)} \leq C \|f\|_{L^p(\mu)}\mbox{ for }f \in L^p(\mu), \ve>0\}.$$
In other words, $T_\mu$ is bounded on $L^p(\mu)$ if and only if $\|T_{\mu}\|_{L^p(\mu) \ra L^p(\mu)} < \infty$.
\end{definition}

The following remark asserts that convolution type SIOs acting on upper $1$-regular sets are invariant under left translations. 
\begin{remark}
\label{translhoriz} Suppose that $K$ is a $1$-dimensional CZ kernel, $E \subset \G$ is  upper $1$-regular and $x \in \G$. Then, for any $p \in (1,\infty)$,
$$\|T_{\mathcal{H}^1|_{E}}\|_{L^p(\mathcal{H}^1|_{E}) \ra L^p(\mathcal{H}^1|_{E}) }=\|T_{\mathcal{H}^1|_{xE}}\|_{L^p(\mathcal{H}^1|_{xE}) \ra L^p(\mathcal{H}^1|_{xE}) }.$$
This follows because $T$ is a convolution type SIO and $\mathcal{H}^1$ is left invariant on $\G$. Indeed, let $f \in L^p(\mathcal{H}^1|_{xE})$ and $\ve>0$. Let also $z=x z'$ where $z' \in E$. Then, after changing variables
$$T_{\mathcal{H}^1|_{xE},\ve}f (z)=\int_{B(z,\ve)^c \cap xE} K(q^{-1} z) f(q) d \mathcal{H}^1(q)=\int_{B(z',\ve)^c \cap E} K(y^{-1}z') f(xy) d \mathcal{H}^1(y) =T_{\mathcal{H}^1|_{E},\ve}g(z'),$$
where $g:E \ra \R$ is defined as $g(y)=f(xy)$. Then since the left invariance of $\mathcal{H}^1$ implies that $\|f\|_{L^p(\mathcal{H}^1|_{xE})}=\|g\|_{L^p(\mathcal{H}^1|_{E})}$, our claim follows.
\end{remark}

In order to eventually apply the $T1$ theorem for homogeneous metric spaces,
we must introduce the adjoint operator $\tilde{T}_\mu$.
\begin{definition}
\label{adjoint}
If $K: \mathbb{G} \setminus \{0\} \to \mathbb{R}$
is a 1-dimensional CZ kernel,
define the \emph{adjoint kernel} $\tilde{K}$ as $\tilde{K}(p) = K(p^{-1})$ for all $p \in \mathbb{G}$.
\end{definition}

Note that 
\begin{align*}
\int (T_{\mu,\varepsilon} f)g \, d \mu
&=
\int \left( \int_{d( p,q) > \varepsilon} K(q^{-1}p) f(q) \, d \mu(q) \right) g(p) \, d \mu (p) \\
&=
\int \left( \int_{d(q,p)  > \varepsilon} \tilde{K}(p^{-1}q) g(p) \, d \mu(p) \right) f(q) \, d \mu (q)
=
\int (\tilde{T}_{\mu,\varepsilon} g)f \, d \mu.
\end{align*}
That is, $\tilde{K}$ is the kernel of the adjoint $\tilde{T}_{\mu,\varepsilon}$ of $T_{\mu,\varepsilon}$.
Note that $\tilde{K}$ will also satisfy the conditions \eqref{CZ1} and \eqref{CZ2} whenever $K$ is itself a 1-dimensional CZ kernel.

\begin{definition}
\label{ubhl}
Given a kernel $K : \mathbb{G} \setminus \{ 0 \} \to \mathbb{R}$
with $| K(p) | \lesssim d (p, 0)^{-1}$,
we say that it is {\em uniformly bounded on horizontal lines} (or {\em UBHL})
if the SIO associated to $K$ is bounded on $L^2(\mathcal{H}^1|_L)$ for any horizontal line $L$ in $\mathbb{G}$
(with constants independent of the choice of $L$).
\end{definition}
A function $\psi:\mathbb{G} \to \mathbb{R}$ is $\mathbb{G}$-\emph{radial}
if there is a function $f:[0,\infty) \to \mathbb{R}$
so that $\psi(p) = f(d( p,0))$ for every $p \in \mathbb{G}$.
Given a $\mathbb{G}$-radial function $\psi$, we write
$$
\psi^r(p) := (\psi \circ \delta_{1/r})(p) 
\quad \text{for all } r>0, \, p \in \mathbb{G}.
$$
\begin{definition}
\label{annubdd}
A kernel $K: \mathbb{G} \setminus \{0\} \to \mathbb{R}$
satisfies the \emph{annular boundedness condition}
if, for every $\mathbb{G}$-radial, $C^{\infty}$ function $\psi:\mathbb{G} \to \mathbb{R}$
satisfying $\chi_{B(0,1/2)} \leq \psi \leq \chi_{B(0,2)}$,
there is a constant $A \geq 1$ (possibly depending on $\psi$)
such that
\begin{equation}
\label{annular}
\left| \int_L [\psi^R(p) - \psi^r(p) ] K(p) \, d\mathcal{H}^1(p) \right| \leq A
\quad \text{for all } 0<r<R
\end{equation}
where $L$ is any horizontal line.
\end{definition}
It follows from Proposition~\ref{p:grouplaw} \eqref{groupl1}
that $d(x,y) = |x-y|$ for any $x,y \in L$,
so, here, $\mathcal{H}^1$ may be the Hausdorff 1-measure associated with either the Euclidean or Carnot metric. We remark that annular boundedness was first introduced in \cite{CFOsios} for $3$-dimensional kernels and vertical planes in the Heisenberg group. 

\begin{remark}
\label{sym+hsym}
A kernel $K$ is said to be \emph{antisymmetric} if $K(p^{-1})=-K(p)$ for any $p \in \mathbb{G} \setminus \{ 0 \}$, 
and we say that $K$ is \emph{horizontally antisymmetric}
if $K(-p_1, p_2, \dots, p_s) = -K(p_1,\dots, p_s)$ for $p_i \in \R^{\dim V_i}, \, i=1,\dots,s,$. 
It is easy to check that any antisymmetric or horizontally antisymmetric kernel necessarily satisfies the annular boundedness condition.
\end{remark}

In the following proposition we will prove that annular boundedness is equivalent to UBHL.


\begin{proposition}
\label{ubhlab}
Let $K$ be a $1$-dimensional CZ kernel. Then $K$ is UBHL if and only if
it satisfies the annular boundedness condition.
\end{proposition}
\begin{proof}
If $K$ is UBHL then arguing exactly as in the proof of  \cite[Lemma~2.9]{CFOsios} we get that the annular boundedness condition is satisfied. Indeed, the metric $d$ is equal to the Euclidean metric along any horizontal line $L$, and the group operation restricted in any horizontal line is simply Euclidean addition.

Now assume that $K$ satisfies the annular boundedness condition. Fix some horizontal line $L$ and let $\mu=\mathcal{H}^1|_{L}$. As noted after Definition \ref{annubdd}, we can take $\mathcal{H}^1$ to be the Euclidean Hausdorff measure (or simply the $1$-dimensional Lebesgue measure). By the preceding observations if $p \in L$ then $d(p,\bo)=|p|$ where $|\cdot|$ denotes the Euclidean norm. We will show that there exists some constant $C$ depending only on $B$ and $A$ such that, for any $0<r<R<\infty$,
\begin{equation}
\label{grafann}
\left| \int_{r<|p|<R} K(p) d\mu(p) \right| \leq C.
\end{equation}

Indeed, for $p \in L$ let $\eta_{r,R}(p)=\psi^R(p) - \psi^r(p)=\psi(p/R)-\psi(p/r)$
where $\psi$ is as in the definition of annular boundedness. 
Note that $\eta_{r,R}(p)=0$ if $|p|>2R$ or $|p|<r/2$. We then write,
\begin{equation*}
\begin{split}
&\left| \int_{r<|p|<R} K(p) d\mu(p) -\int_{L} \eta_{r,R}(p) K(p) d\mu(p) \right| \\
&\quad\quad\quad\leq \int_{r/2<|p|<2r} |K(p)| d\mu(p)+\int_{R/2<|p|<2R} |K(p)|d\mu(p)+ \left| \int_{2r<|p|<R/2} K(p)(1-\eta_{r,R}(p)) d\mu(p) \right| \\
&\quad\quad\quad:= I_1+I_2+I_3.
\end{split}
\end{equation*}
If $2r \geq R/2$ then trivially $I_3=0$. By the size estimate \eqref{CZ1} of $K$, we have $I_1 +I_2 \leq 8B$. On the other hand, if 2r<|p|<R/2, then $n_{r,R}(p)=1$. Hence, $I_3=0$ and we obtain \eqref{grafann} with $C=8B+A$. Note that $K$ restricted to $L$ is a $1$-dimensional Euclidean kernel, indeed if $L=\{ (sx_1,\dots,sx_n,0,\dots,0) \, : \, s \in \R \}$ then we can identify any $p=(sx_1,\dots,sx_n,\bo) \in L$ with $s$ and define the kernel $\tilde{K}(s):=K(sx_1,\dots,sx_n,\bo)$. Now the result follows from \cite[Theorem 5.4.1]{grafakos} upon noticing that $\tilde{K}$ satisfies properties \cite[(5.4.1)-(5.4.3)]{grafakos}. See also \cite[p.374--5]{grafakos}.
\end{proof}

\subsubsection{Examples} We now present two families of $1$-dimensional CZ kernels which satisfy the annular boundedness conditions, and hence Theorem~\ref{main} may be applied to these kernels. In the following examples we will also assume that the left invariant, homogeneous metric $d$ has been chosen so that $p \mapsto d(p,0)$ is of class $C^1$ on $\mathbb{G} \setminus \{0\}$.
(Such a choice can always be made.)

\begin{example}
\label{vriesz} The \emph{vertical $\G$-Riesz kernels} are defined by
$$V_n(p)=\frac{d(NH(p),0)^n}{d(p,0)^{n+1}}, \quad p \in \G \stm  \{0\}, n \in \mathbb{N},$$
where $NH(p)=\tilde{\pi}(p)^{-1}p$ is the \emph{non-horizontal part} of $p$. The vertical $\G$-Riesz kernels are $1$-dimensional CZ kernels. Indeed, note first that the size condition \eqref{CZ1} is satisfied because $d(NH(p),0) \lesssim d(p,0)$. Moreover, we have  that $V_n(p)$ is $-1$-homogeneous 
(i.e. $V_n(\delta_r(p))=r^{-1} V_n(p)$ for $r>0$ and $p \in \G \stm  \{0\}$) and of class $C^1$ on $\G \stm  \{0\}$. Hence, from \cite[Proposition 1.7]{stfol} we infer that 
\begin{equation}
\label{vnhold}|V_n(qp)-V_n(q)| \lesssim d(p,0) d(q,0)^{-2} \mbox{ and }|V_n(pq)-V_n(q)|  \lesssim d(p,0) d(q,0)^{-2}
\end{equation} for $d(p,0) \leq 2^{-1} d(q,0)$. Although the second inequality cannot be deduced directly from \cite[Proposition 1.7]{stfol}, we obtain it by using the smoothness of the map $p \mapsto pq$ and arguing exactly as in \cite[Proposition 1.7]{stfol}. Now \eqref{vnhold} easily implies the H\"older condition \eqref{CZ2}, see also \cite[Lemma 2.7]{ChoLi}. Moreover, since the kernels $V_n$ vanish on horizontal lines they satisfy the annular boundedness condition.

The kernels $V_n$, were first considered in \cite{ChoLi} in the first Heisenberg group. It was shown there that, if $\Gamma$ is a $1$-regular curve, then the kernel $V_8$ defines a SIO which is $L^2(\mathcal{H}^1|_{\Gamma})$-bounded. This result was generalized to arbitrary Carnot groups in \cite{ChoLiZimTra} for symmetrizations of $V_{2s^2}$, where $s$ is the step of the group. Conversely, it was also proved in \cite{ChoLi} that, if $E$ is a $1$-regular subset of the Heisenberg group and $V_2$ defines a SIO which is $L^2(\mathcal{H}^1|_{E})$-bounded, then $E$ is contained in a $1$-regular curve. These were the first non-Euclidean examples of kernels with such properties. In addition, unlike in the Euclidean case where all known kernels with such properties are antisymmetric, the kernels $V_n$ are nonnegative and are symmetric (for Carnot groups of step 2) or can by symmetrized (as in \cite{ChoLiZimTra}).
\end{example}

\begin{example} The $1$-dimensional \emph{quasi $\G$-Riesz kernel} is defined by
$$
\Omega(p)= \left( \frac{p_1}{d(p,0)^2}, \frac{p_2}{d(p,0)^3}, \dots, \frac{p_s}{d(p,0)^{s+1}} \right),
$$
where $p=(p_1, \dots, p_s) \in \G \stm \{0\}$ for $p_i \in \R^{\dim V_i}$, $i=1,\dots,s$. Note that the kernel $\Omega$ is $-1$-homogeneous. Hence, arguing as in Example~\ref{vriesz}, we can see that the coordinates of the quasi $\G$-Riesz kernel are $1$-dimensional CZ kernels. Moreover, $\Omega$ is antisymmetric, so recalling Remark~\ref{sym+hsym}, it also satisfies the annular boundedness condition. 

The kernel $\Omega$, which is modeled after the Euclidean Riesz kernels, was introduced in \cite{ChoMat} for the Heisenberg groups $\mathbb{H}^n$. It was proved there that, if $\mu$ is an $m$-regular measure for $m \in \mathbb{N} \cap [1, 2n+1]$ and the SIO associated with the $m$-dimensional analogue of $\Omega$ is bounded in $L^2(\mu)$, then $\supp (\mu)$ can be approximated at $\mu$-almost every point and at arbitrary small scales by homogeneous subgroups.

\end{example}

\section{Tools for the proof of Theorem~\ref{main}}
\label{sec-tools}
In this section we will state and prove the two important propositions which will be necessary in the proof of the main theorem.
The first is an application of Lemma~\ref{montilem} to prove a change of variables formula for integrals along rectifiable curves in $\mathbb{G}$.
\begin{proposition}
\label{areaform}
Suppose $\gamma:[a,b] \to \mathbb{G}$ is Lipschitz. Write $\gamma(t) = (\gamma_1(t),\gamma_2(t)) \in \mathbb{R}^n \times \mathbb{R}^{N-n}$
and $G = \gamma([a,b])$.
Then for every $\eta \in L^1(\mathcal{H}^1|_G)$,
we have
$$
\int_{G} \eta \, d \mathcal{H}^1|_G
=
\int_a^b \eta(\gamma(t)) | \gamma_1'(t) | \, dt.
$$
\end{proposition}
\begin{proof}
Define the speed $|\dot{\gamma}|$ of $\gamma$ as
$$
|\dot{\gamma}|(t) 
:= \lim_{s \to 0} \frac{d(\gamma(t+s),\gamma(t))}{|s|}
= \lim_{s \to 0} d( \delta_{1/|s|} (\gamma(t)^{-1}\gamma(t+s)), 0)
$$
for every $t \in [a,b]$ for which this limit exists.
In particular, Lemma~\ref{montilem} and Proposition~\ref{p:grouplaw} {\em (1)} give
$$
|\dot{\gamma}|(t) 
= \lim_{s\searrow 0} d \left( \delta_{1/|s|} (\gamma(t)^{-1}\gamma(t+s)),0 \right) 
= d \left( (\gamma_1'(t),0), 0 \right)
= | \gamma_1'(t) |
\quad \text{ for a.e. } t \in [a,b].
$$
Let $I \subset [a,b]$ be any open interval.
According to \cite[Theorem 3.6]{HajSobMet} then, we have
$$
\mathcal{H}^1(\gamma(I))
= \ell(\gamma|_I) 
= \int_I |\dot{\gamma}|(t) \, dt
= \int_I | \gamma_1'(t) | \, dt.
$$
Using standard approximation arguments of $L^1$ functions by characteristic functions $\chi|_{\gamma(I)}$, 
we have proven the proposition.
\end{proof}

As seen in Lemma~\ref{montilem},
Pansu's Theorem implies that a Lipschitz curve in $\mathbb{G}$ 
is well approximated by its horizontal tangent lines.
According to the next proposition,
if such a Lipschitz curve is further assumed to be $C^{1,\alpha}$, 
then this approximation by horizontal tangents is quantitatively strong.
\begin{proposition}
\label{flat}
Suppose $\gamma:[0,1] \to \mathbb{R}^N$ is of class $C^{1,\alpha}$
and is Lipschitz in $\mathbb{G}$.
Write 
$\gamma = (\gamma_1,\dots,\gamma_N)$, 
fix $t_0 \in [0,1]$,
and set 
$$
L(t) := \gamma(t_0)*((t-t_0)\gamma_1'(t_0),\dots,(t-t_0)\gamma_n'(t_0),0,\dots,0)
\quad
\text{for }
t \in \mathbb{R}.
$$
Then
\begin{equation}
\label{eqpromonti}
d(\gamma(t),L(t)) \lesssim |t-t_0|^{1+\frac{\alpha}{s}}
\quad
\text{for all } t \in [0,1].
\end{equation}
(The constant in this bound necessarily depends on the choice of $\gamma$.)
\end{proposition}

\begin{proof}
Since the metric $d$ is invariant under left translation,
we may assume without loss of generality that $\gamma(t_0) = 0$.
Also, by the symmetry in the arguments below, we may assume $t_0 = 0$.
We write the group operation in $\mathbb{G}$ as $xy = x + y + Q(x,y)$ for a polynomial $Q$
with the properties outlined in Proposition~\ref{p:grouplaw}.

Since $d(\gamma(t),L(t)) \lesssim d_{\infty}(\gamma(t),L(t)) = \Vert L(t)^{-1} \gamma(t) \Vert$,
it suffices to establish for every $t \in [0,1]$ the bounds
\begin{equation}
\label{firstlayergoal}
\left| \gamma_i(t) - t \gamma_i'(0)\right| \lesssim t^{1+\alpha} \leq t^{1 + \frac{\alpha}{s}}
\end{equation}
when $1 \leq i \leq n$ and
\begin{equation}
\label{secondlayergoal}
\left| \gamma_i(t) + Q_i(L(t)^{-1},\gamma(t)) \right| \lesssim t^{d_i+\alpha} \leq t^{d_i(1 + \frac{\alpha}{s})}
\end{equation}
when $n < i \leq N$
where $d_i \leq s$ is the degree of the $i$th coordinate.
Indeed, this follows from the definitions of $L(t)$ and the norm $\Vert \cdot \Vert$.

Fix $t \in [0,1]$.
Fix $1 \leq i \leq n$.
The H\"{o}lder continuity of $\gamma'$ gives
$$
\left| \gamma_i(t) - t \gamma_i'(0)\right|
\leq
\int_0^t |\gamma_i'(s) - \gamma_i'(0)| \, ds
\lesssim 
\int_0^t s^{\alpha} \, ds
\leq t^{1+\alpha}
$$
which proves \eqref{firstlayergoal}.

Now fix $n < i \leq N$.
Following the example of Monti in \cite{MontiThesis},
write $h = (\gamma_1',\dots,\gamma_n',0,\dots,0)$.
(Note that, by definition, 
$\gamma'(0) = h(0)$ since $\gamma(0) = 0$.)
We will first establish a bound in \eqref{secondlayergoal} for $|\gamma_i(t)|$.
Suppose by way of induction that we have shown 
$$
\left|\gamma_{j}(s) - s^{d_j} \gamma_{j}'(0)\right|
=
\left|\gamma_{j}(s) - s^{d_j} h_{j}(0)\right| 
\lesssim s^{d_{j}+\alpha}
$$
for every $j < i$ and $s \in [0,1]$.

As in equation (1.7.83) and the proof of Lemma~ 2.1.4 
in \cite{MontiThesis}, we can write 
$$
\gamma_i'(s) = \sum_{j=1}^n \gamma_j'(s) \left. \frac{\partial Q_i(\gamma(s),y)}{\partial y_j} \right|_{y=0} = \bar{Q}_i(\gamma(s),h(s))
$$
for every $s \in [0,1]$
where $\bar{Q}_i(x,y)$ is the finite sum of the monomials in $Q_i$ in which $y$ appears linearly.
Note that $Q_i(x,y)$ depends only on $x_\ell$ and $y_\ell$ with $\ell < i$ (see Proposition~\ref{p:grouplaw} \eqref{groupl4}).
Since it follows from Proposition~\ref{p:grouplaw} {\em (3)} that $\bar{Q}_i$ is homogeneous of degree $d_i$,
and since
$\delta_{1/s}(h(s)) = h(s) / s$,
we may conclude that 
$$
s^{1-d_i}\gamma_i'(s) = s^{1-d_i} \bar{Q}_i(\gamma(s),h(s)) = \bar{Q}_i(\delta_{1/s}(\gamma(s)),h(s)).
$$

Moreover, Proposition~\ref{p:grouplaw} {\em (2)} implies that
$Q_i(x,y)$ is a sum of terms 
each of which contains a factor of the form
$(x_j y_\ell - x_\ell y_j)$ for some $1 \leq j,\ell < i$.
That is, we can write each such term as
$$
p(x,y) (x_j y_\ell - x_\ell y_j)
$$
where $p(x,y)$ is some polynomial.
Since $y$ appears linearly in the monomials of $\bar{Q}_i$,
we may conclude that $\bar{Q}_i$
is also a finite sum of terms of the form 
$$
\bar{p}(x) (x_j y_\ell - x_\ell y_j)
$$
for some $1 \leq j,\ell < i$ 
where $\bar{p}(x)$ is a polynomial.
In particular, 
this implies that 
$
s^{1-d_i} \gamma_i'(s) 
$
is a finite sum of terms each of which possesses 
$$
s^{-d_j} \gamma_j(s) h_\ell(s) - s^{-d_\ell}\gamma_\ell(s) h_j(s)
$$ 
as a factor for some $1 \leq j,\ell < i$.
By the induction hypothesis and the continuity of $\gamma'$ on $[0,1]$, 
each such term may therefore be bounded by a constant multiple of
\begin{align*}
|s^{-d_j} \gamma_j(s) h_\ell(s) - s^{-d_\ell} \gamma_\ell(s) h_j(s)|
&\leq s^{-d_j} | \gamma_j(s) - s^{d_j} h_j(0)|| h_\ell(s)| 
+ 
s^{-d_\ell}|\gamma_\ell(s) - s^{d_\ell} h_\ell(0)|| h_j(s)| \\
& \hspace{2in}	+ |h_j(0) h_\ell(s) - h_\ell(0) h_j(s)| \\
&\lesssim s^{\alpha}.
\end{align*}
Indeed, the bound on the last term follows from
the H\"{o}lder continuity of $\gamma'$ since
$$
|\gamma_j'(0)\gamma_\ell'(s) - \gamma_\ell'(0) \gamma_j'(s)| 
\leq 
|\gamma_j'(0)||\gamma_\ell'(s) - \gamma_\ell'(0)|
+
|\gamma_\ell'(0)|| \gamma_j'(0) - \gamma_j'(s)|
\lesssim s^{\alpha}.
$$
We therefore have
$$
t^{-d_i} |\gamma_i(t)|
\leq
\frac{1}{t} \int_0^t |s^{1-d_i} \gamma_i'(s)| \, ds
=
\frac{1}{t} \int_0^t |\bar{Q}_i(\delta_{1/s}(\gamma(s)),h(s))| \, ds
\lesssim
\frac{1}{t} \int_0^t s^\alpha \, ds
\leq
t^{\alpha}
$$
for any $t \in [0,1]$.
This completes the induction step, and thus
\begin{equation}
\label{goal1}
|\gamma_i(t) - t^{d_i} h_i(0)| = |\gamma_i(t)| \lesssim t^{d_i+\alpha}
\end{equation}
for all $n < i \leq N$ and $t \in [0,1]$.

We will now establish a bound for $|Q_i(L(t)^{-1},\gamma(t))|$.
Recall from Proposition~\ref{p:grouplaw} \eqref{groupl3} that 
$$
t^{-d_i} Q_i(L(t)^{-1},\gamma(t)) = Q_i(\delta_{1/t}(L(t)^{-1}),\delta_{1/t}(\gamma(t))).
$$
As discussed above, we may write the polynomial $Q_i(\delta_{1/t}(L(t)^{-1}),\delta_{1/t}(\gamma(t)))$ 
a sum of terms 
each of which contains a factor of the form 
$$
t^{-d_\ell} \gamma_\ell(t) \gamma_j'(0)   -  t^{-d_j} \gamma_j(t) \gamma_\ell'(0) 
$$ 
for some $1 \leq j,\ell < i$
(noting that $\gamma_m(k)(0)=0$ when $k > n$).
Therefore,
we may appeal to the inequalities proven above 
to bound each summed term in the polynomial by a constant multiple of
\begin{align*}
\left|t^{-d_\ell} \gamma_\ell(t) \gamma_j'(0)   -  t^{-d_j} \gamma_j(t) \gamma_\ell'(0)  \right|
&\leq
t^{-d_j}
\left|\gamma_j(t) - t^{d_j}\gamma_j'(0) \right|
\left| t^{-d_\ell} \gamma_\ell(t) \right| \\
& \hspace{.7in}
+
t^{-d_\ell}
\left|\gamma_\ell(t) - t^{d_\ell}\gamma_\ell'(0) \right|
\left| t^{-d_j} \gamma_j(t) \right|
\lesssim
t^{2\alpha} \leq t^{\alpha}.
\end{align*}
In other words, we have
$$
\left|Q_i(L(t)^{-1},\gamma(t))\right| 
\lesssim 
t^{d_i+\alpha}.
$$
This together with \eqref{goal1} verifies \eqref{secondlayergoal}
and completes the proof of the proposition.
\end{proof}

\section{Proof of Theorem \ref{main}}
\label{sec-main}
Before starting the proof of Theorem \ref{main} we need an auxiliary lemma. As usual we define the centered Hardy-Littlewood maximal function $M_\mu$ associated with a Radon measure $\mu$ on $\G$, by
\begin{equation}
\label{hlmax}
M_\mu f (x)=\sup_{r>0} \frac{1}{\mu(B(x,r))} \int_{B(x,r)} |f(y)| d \mu (y), \quad f \in L^{1}_{loc}(\mu).
\end{equation} It is well known, see e.g. \cite[Theorem 3.5.6]{HeiKosShaTys}, that if $\mu$ is doubling, then for all $p \in (1, \infty]$ there exist constants $c_p$, only depending on $p$ and the regularity constant of $\mu$, such that $\|M_\mu f\|_{p} \leq c_p \|f\|_p$ for $f \in L^p(\mu)$.

\begin{lemma}
\label{sepsupports}
Let $K: \mathbb{G} \setminus \{0\} \to \mathbb{R}$ be a continuous function satisfying \eqref{CZ1}. Let also $\mu$ be a $1$-regular, measure with regularity constant $C_R$, such that $\supp(\mu)=A \cup B$, where $B$ is bounded and $\dist(A,B) \geq \diam(B)$. Let $\mu_1=\mu|_{A}$ and $\mu_2=\mu|_{B}$. If 
\begin{equation}
\label{tm1m2bd}
\|T_{\mu_1}\|_{L^2(\mu_1) \ra L^2(\mu_1)} \leq C_1<\infty \mbox{ and }\|T_{\mu_2}\|_{L^2(\mu_2) \ra L^2(\mu_2)} \leq C_2<\infty,
\end{equation}
then $$\|T_{\mu}\|_{L^2(\mu) \ra L^2(\mu)}\leq C(C_1,C_2,C_R,K)<\infty.$$
\end{lemma}
\begin{proof} Let $f \in L^2(\mu)$ and $\ve>0$. Then
$$\|T_{\mu,\varepsilon}\|^2_{L^2(\mu)}=\int_A |T_{\mu, \ve}f(x)|^2 d \mu(x)+\int_B |T_{\mu, \ve}f(x)|^2 d \mu(x):=I_1+I_2.$$
We first treat $I_1$:
\begin{equation}
\label{i1}
\begin{split}
I_1&=\int_A \left| \int_{B(x,\ve)^c} K(y^{-1}x) f(y) d \mu (y)\right|^2 d \mu(x) \\
&\lesssim \int_A \left| \int_{B(x,\ve)^c \cap A} K(y^{-1}x) f(y) d \mu (y)\right|^2 d \mu(x) +\int_A \left| \int_{B(x,\ve)^c \cap B} K(y^{-1}x) f(y) d \mu (y)\right|^2 d \mu(x) \\
&\overset{\eqref{tm1m2bd}}{\leq} C_1 \|f\|^2_{L^2(\mu)}+\int_A \left| \int_{B(x,\ve)^c \cap B} K(y^{-1}x) f(y) d \mu (y)\right|^2 d \mu(x) \\
&:=C_1 \|f\|^2_{L^2(\mu)}+I_{12}.
\end{split}
\end{equation}
Let 
$$g(x)=\int_{B(x,\ve)^c \cap B} K(y^{-1}x) f(y) d \mu (y).$$
Then for $x \in A$, by our assumption $d(x,B) \geq d(A,B) \geq \diam(B)$. Therefore, 
\begin{equation}
\label{gbdbymax}|g(x)| \overset{\eqref{CZ1}}{\lesssim} \int_{B} d(x,y)^{-1} |f(y)| d \mu(y) \leq \frac{1}{d(x,B)} \int_{B(x, 2d(x,B))} |f(y)| d \mu(y)  \lesssim M_{\mu} f(x).
\end{equation}
Thus, by the $L^2(\mu)$-boundedness of $M_\mu$,
\begin{equation}
\label{i12}
I_{12} \lesssim \int_A M_\mu f (x)^2 d \mu (x) \lesssim \|f\|^2_{L^2(\mu)} 
\end{equation}
and  $I_1 \overset{\eqref{i1} \wedge\eqref{i12} }{\lesssim} \|f\|^2_{L^2(\mu)}.$

We now estimate $I_2$ as in \eqref{i1},
\begin{equation}
\label{i2}
\begin{split}
I_2&\overset{\eqref{tm1m2bd}}{\leq} C_2 \|f\|^2_{L^2(\mu)}+\int_B \left| \int_{B(x,\ve)^c \cap A} K(y^{-1}x) f(y) d \mu (y)\right|^2 d \mu(x) \\
&:=C_2 \|f\|^2_{L^2(\mu)}+I_{22}.
\end{split}
\end{equation}
Let
$$h(x)=\int_{B(x,\ve)^c \cap A} K(y^{-1}x) f(y) d \mu (y).$$
For $x \in B$, we have by Cauchy-Schwartz,
\begin{equation}
\label{cs}
\begin{split}
|h(x)|^2 &\leq \left( \int_{B(x,\ve)^c \cap A} |K(y^{-1}x)|^2 d \mu (y)  \right)  \left( \int |f(y)|^2 d \mu (y)  \right)\\
&\overset{\eqref{CZ1}}{\lesssim} \left( \int_{\{y: d(x,y)>\diam(B)\}} d(x,y)^{-2} d \mu(y) \right) \|f\|^2_{L^2(\mu)} \\
& \lesssim \diam(B)^{-1} \|f\|^2_{L^2(\mu)},
\end{split}
\end{equation}
where in the last inequality we split the integral on annuli, as in \cite[Lemma 2.11]{tolsabook}. Therefore, since $\mu$ is $1$-regular and $B$ is bounded, 
\begin{equation}
\label{i22}
I_{22} \overset{\eqref{cs}}{\lesssim} \int_{B}\diam(B)^{-1} \|f\|^2_{L^2(\mu)} d \mu(x) \lesssim   \|f\|^2_{L^2(\mu)}.
\end{equation}
Hence $I_1 \overset{\eqref{i2} \wedge\eqref{i22} }{\lesssim} \|f\|^2_{L^2(\mu)}.$ The proof is complete.
\end{proof}

\begin{proof}[Proof of Theorem~\ref{main}]

Suppose $\Gamma$ is a $C^{1,\alpha}$ regular curve,
and $\gamma:[0,1] \to \mathbb{G}$ is its $C^{1,\alpha}$ arc length parameterization.
Write $C_R := \reg(\mathcal{H}^1|_\Gamma)$.
We may assume without loss of generality that the arc length of $\Gamma$ is 1, and moreover, by Remark \ref{translhoriz}, we can also assume that $\gamma(0)=0$.
It follows from  \cite[Proposition~5.1.8]{HeiKosShaTys} and the proof of Proposition~\ref{areaform}
that $\gamma$ is 1-Lipschitz
and $|\gamma'| \geq |\gamma_1'| = 1$.
Thus, since $\gamma$ is $C^1$,
there is some $0 < \delta < \tfrac15$ so that 
$$
|t_2 - t_1| |\gamma'(t_1)| - |\gamma(t_2) - \gamma(t_1)| <  \tfrac{1}{2} |t_2 - t_1|
$$
for any $|t_2-t_1| < \delta$ in $[0,1]$,
and hence 
\begin{equation}
\label{radius}
|t_2-t_1| \leq 2|\gamma(t_2) - \gamma(t_2)| \leq D \,  d(\gamma(t_2),\gamma(t_1))
\quad
\text{for } 
t_1,t_2 \in [0,1],
\; |t_1 - t_2| <\delta
\end{equation}
where $D/2>1$ is the constant given by Proposition~\ref{compact} depending only on $\mathbb{G}$ and $\text{diam}(\Gamma)$.

Let $K:\mathbb{G} \setminus \{0\} \to \mathbb{R}$ be a 1-dimensional CZ kernel satisfying
the UBHL condition, and denote by $T_{\mathcal{H}^1|_{\Gamma}}$ the SIO associated with $K$. 
Arguing as in \cite[Lemma 3.1]{CFOsios} it suffices to prove that $T_{\mathcal{H}^1|_{\Gamma}}$ is $L^p(\mathcal{H}^1|_{\Gamma})$-bounded. We will do so by using an appropriate  ``good $\lambda$ method''. 

For the next proposition we define $\Sigma$ to be the set of all Radon measures $\nu$ on $\G$ which satisfy \eqref{meas-reg} and $\diam(\supp(\nu))=\infty$. 

\begin{proposition}
\label{goodl}
Let $\nu \in \Sigma$ and suppose that there exist constants $0< \theta <1,  C \geq 1$ and $C_p > 0, p \in (1,\infty)$, such that
for every $B = B(x,r)$ with $x \in \supp (\nu)$ and $r>0$,
there is a compact set $G \subset B \cap \supp (\nu) $ and a Radon measure $\sigma \in \Sigma$ such that
\begin{enumerate}
\item \label{gl1}$\reg(\sigma) \leq C $,
\item \label{gl2} $\nu(G) \geq \theta \nu(B)$,
\item  \label{gl3} $\nu(A \cap G)\leq  \sigma(A)$ for all $A \subset \G$,
\item  \label{gl4} $\Vert T_{\sigma,*} f \Vert_{L^{p}(\sigma)} \leq C_p \Vert f \Vert_{L^p(\sigma)}$ for all $1 < p < \infty$.
\end{enumerate}
Then there exist constants $A_p=A_p(C_p, C, \reg(\nu), K, \theta),p \in (1,\infty),$ such that 
$\Vert T_{\nu,*} f \Vert_{L^{p}(\nu)} \leq A_p \Vert f \Vert_{L^p(\nu)}$.
\end{proposition}

Proposition \ref{goodl} follows from \cite[Proposition 3.2, p60]{david-wavelets}.  While the setting in \cite[Proposition 3.2, p60]{david-wavelets} is Euclidean, its proof can be applied in our case with only minor modifications. In particular, David uses the Besicovitch covering theorem in the proof of \cite[Lemma 2.2]{david-wavelets}, but one can do away with this issue by applying the $5r$-covering lemma. This modifications (along with several other subtler ones which do not arise in our setting) has been treated in detail in \cite[Theorem 6.3]{FOcurves} where the authors extend \cite[Proposition 3.2, p60]{david-wavelets} to metric spaces and generalized CZ kernels.

Note that the hypothesis of Proposition \ref{goodl}
requires the support of the measure $\nu$
to have infinite diameter,
but the support of $\mathcal{H}^1|_\Gamma$ is compact.
We rectify this issue with the following construction. Choose any unit vector $v_0 \in \R^n$
and consider the horizontal ray 
$$
\ell_0= \{(sv_0,\bo) \, : \, s \in [3,\infty)\} \subset \mathbb{R}^n \times \{ \bo \}.
$$
Note that
$
d(\ell_0,\Gamma)  > 1 = \mathcal{H}^1(\Gamma).
$
We set 
$$\tilde{\Gamma} = \Gamma \cup \ell_0,$$ and we let $\nu=\mathcal{H}^1|_{\tilde{\Gamma}}$. Note, that $\nu \in \Sigma$, and $\reg(\nu)$ only depends on $C_R$, which is the regularity constant  of $\Gamma$.

Let us see how to choose a ``nice'' set $G$ inside any ball centered on $\tilde{\Gamma}$. Fix $x \in \tilde{\Gamma}$ and $r>0$, and set $B = B(x,r)$. In the following, $\theta := \delta(2 D C_R^2(1+ d(\ell_0,\Gamma))^{-1}$.

\noindent \textbf{Case 1}: $x \in \ell_0$ and $r >0$

Set $G = B \cap \ell_0$ so that $\nu(G) \geq r$.
If $r \geq \nu(\Gamma)$,
then
$$
\nu(B)=\nu(B \cap \tilde{\Gamma})
\leq \nu(\Gamma) + \nu(B \cap \ell_0)
\leq r + 2r
\leq 3 \nu(G).
$$
If $r < \nu(\Gamma)$, then $B \cap \Gamma = \emptyset$, so
$
\nu(B)=\nu(B \cap \tilde{\Gamma}) = \nu(B \cap \ell_0) = \nu(G)
$.

\noindent \textbf{Case 2}: $x \in \Gamma$ and $r \geq 2(1 + d(\ell_0,\Gamma))$

Note that $B((x_0, \bo),\frac{r}{2}) \subset B$ since $\nu(\Gamma) = 1$. Choosing $G = B((x_0,\bo),\frac{r}{2}) \cap \ell_0 \subset B$ gives $\nu(G) = \frac{r}{2}$.
Thus
$$
\nu(B) \leq \nu(\Gamma) + \nu(B \cap \ell_0) \leq \tfrac{r}{2} + 2r = 5 \nu(G).
$$

\noindent \textbf{Case 3}: $x \in \Gamma$ and $r < 2(1 + d(\ell_0,\Gamma))$

Choose $a,b \in [0,1]$ so that 
$\gamma(a) = x$
and 
$
|b-a| = \frac{\delta r}{2C_R (1 + d(\ell_0,\Gamma))} < \delta.
$
(Without loss of generality,
we may assume that $a < b$.)
Set $G := \gamma([a,b])$.
Note first that $d(\gamma(t),\gamma(a)) \leq |b-a| < r$ for all $t \in [a,b]$,
so $G \subset B$.
Moreover, the bound \eqref{radius} 
and the regularity of $\Gamma$ give
$$
\nu(G)
\geq d(\gamma(a),\gamma(b)) 
\geq \tfrac{1}{D} |b-a| 
= \frac{\delta r}{2D C_R (1 + d(\ell_0,\Gamma))}
\geq \frac{\delta}{2D C_R^2 (1 + d(\ell_0,\Gamma))} \nu(B)
= \theta \nu(B).
$$
Note \eqref{radius} implies that $G$ is the bi-Lipschitz image of an interval
with bi-Lipschitz constant $D$.
In particular, $G$ is a regular curve 
with a regularity constant independent of the choice of $G$,
and
\begin{equation}
\label{smallballs} \text{diam} \left(\gamma^{-1}(B(x,r) \cap G) \right)\leq
D r
\quad
\text{for all } x\in G, \, r > 0.
\end{equation}

Given any ball $B(x,r)$ with $x \in\supp (\nu)$, we have chosen a set $G=G_{x,r}$. We now define the measure $\sigma := \sigma_{x,r}$. If $x$ and $r$ are as in Cases 1 and 2, then $G \subset \ell_0$ and we set $\sigma=\mathcal{H}^1 |_{\ell_0}$. 
Clearly, $\sigma \in \Sigma$ and $\reg(\sigma) = 2$. 
On the other hand, if $x$ and $r$ are as in Case 3, then $G \subset \Gamma$. 
Note once again that the diameter of the support of $\sigma$ must be infinite. 
We will now describe how to choose $\sigma$ in this case.

Suppose $x$ and $r$ are as in Case 3. 
Define the horizontal ray
$$\ell_{G}=\{(s v_0, \bo): s \in [3 \diam(G), \infty)\},$$
where $v_0 \in \R^n$ is a unit vector. We set $L_G=x \ell_{G}$. Note that $2\diam (G) \leq d(L_G, G) \leq 3 \diam (G)$ since $x \in G$. We define $\tilde{G}=G \cup L_G$ and $\sigma=\mathcal{H}^1 |_{\tilde{G}}$. Observe that $\sigma \in \Sigma$, and, since $L_G$ is a controlled distance away from $G$, it follows that $\reg(\sigma)$ depends only on the regularity constant $C_R$ of $\Gamma$. Moreover, $\nu(A \cap G)=  \sigma(A)$ for all $A \subset \G$. Therefore, our choices of $G$ and $\sigma$ satisfy \eqref{gl1}, \eqref{gl2}, and \eqref{gl3} of Proposition \ref{goodl}.

We will now verify Proposition~\ref{goodl} \eqref{gl4} i.e. we will show that $T_{\sigma,*}$ is $L^p(\sigma)$-bounded for any $1 < p < \infty$
with constants only dependent on $\gamma, K, \G $, and $p$. In particular, the constants will be independent of the choice of $G$ and $\sigma$. To this end, it will suffice to show that 
\begin{equation}
\label{l2bound}
\|T_{\sigma}\|_{L^2(\sigma) \ra L^2(\sigma)}=C(\gamma, K, \G)<\infty.
\end{equation}
 Indeed, once this is achieved, \cite[Theorem 2.4, p.74]{CW} (see also \cite[Theorem 9, p.94]{christbook} and  \cite[Theorem 2.21]{tolsabook}) will allow us to deduce that $T_\sigma$ is bounded in $L^p(\sigma), p \in (1,\infty),$ and in weak $L^1(\sigma)$ with bounds only depending on $\gamma, K, \G $, and $p$. Then, using Cotlar's inequality as in \cite[Lemma 20.25]{Mat} (wherein the Lemma is stated only for Euclidean spaces but the proof translates without issue to our setting),  we infer that $T_{\sigma,\ast}$ is bounded on $L^p(\sigma), p \in (1,\infty),$ with bounds only depending on $\gamma, K, \G $, and $p$. Note that, for the last step, we could also use the version of Cotlar's Lemma stated in \cite[Theorem 7.1]{NTVcotlar}, but this is rather an overkill since the measures discussed in \cite{NTVcotlar} merely require polynomial growth.

In Cases 1 and 2 above, $\sigma=\mathcal{H}^1|_{\ell
_0}$, and hence  \eqref{l2bound} follows from the UBHL condition. Thus we are left with Case 3. Fix a set $G$ and measure $\sigma$ as in Case 3. We first note that, by the UBHL condition and Remark \ref{translhoriz},
\begin{equation}
\label{horizpartG}
\|T_{\mathcal{H}^1|_{L_{G}}}\|_{L^2(\mathcal{H}^1|_{L_{G}}) \ra L^2(\mathcal{H}^1|_{L_{G}})}\leq C(K)<\infty.
\end{equation}
Set $\mu:=\sigma|_{G}=\nu|_{G}=\mathcal{H}^1|_{G}$. 
In the remainder of the proof, we will show that
\begin{equation}
\label{t1G}
\|T_{\mu}\|_{L^2(\mu) \ra L^2(\mu)}\leq C(K, \G, \gamma)<\infty.
\end{equation}

We start by recalling the so called \textit{Christ cubes} which were introduced in \cite{Chr} and provide decompositions of spaces of homogeneous type much like the usual dyadic cube tiling of Euclidean space. In particular, Christ's theorem applied to $(G,d,\mu)$, reads as follows:


\begin{theorem}
For each $j \in \mathbb{Z}$, there is a family $\Delta_j$ of disjoint open subsets of $G$ satisfying
\begin{enumerate}
\item $G = \bigcup_{Q \in \Delta_j} \overline{Q}$, \label{davuni}
\item if $k \leq i$ and $Q \in \Delta_i$ and $Q' \in \Delta_k$, then either $Q \cap Q' = \emptyset$ or $Q \subset Q'$,
\item if $Q \in \Delta_j$, then $\diam (Q) \leq 2^{-j}$, \label{davdiam}
\item There is a constant $c_{o}>0$ (depending only on $C_R$) so that, for each $Q \in \Delta_j$, there is some point $z_Q \in Q$ so that $B(z_Q,c_{o}2^{-j}) \cap G \subset Q$, \label{davball}
\item There is a constant $C_{\partial}\geq 1$ (depending only on $C_R$)
so that for any $\rho>0$ and any $Q \in \Delta_j$, 
\begin{equation*}
\mu \left( \{ q \in Q \, : \, d(q,G \setminus Q) \leq \rho 2^{-j} \} \right) \leq C_{\partial} \rho^{1/C_{\partial}} \mu (Q). \label{thinbdry}
\end{equation*}
\end{enumerate}
\end{theorem}
Write $\Delta := \bigcup_j \Delta_j$. According to the $T1$ theorem of David and Journ\'{e} \cite{DavJou}
applied to the homogeneous metric measure space $(G,d,\mu)$, in order to show that $T$ is bounded in $L^2(\mu)$ with bounds independent of $G$, it suffices to prove that there exists some constant $C:=C(\gamma, K, \G) \geq 1$ such that, for any $Q \in \Delta$,
\begin{equation}
\label{T1full}
\Vert T_{\mu,\varepsilon} \chi_Q \Vert_{L^2(\mu|_Q)}^2 \leq C \mu(Q) \mbox{ and }\Vert \tilde{T}_{\mu,\varepsilon} \chi_Q \Vert_{L^2(\mu|_Q)}^2 \leq C \mu(Q),
\end{equation}
where $\tilde{T}_{\mu,\varepsilon}$ is the formal adjoint of $T_{\mu,\varepsilon}$ (recall Definition~\ref{adjoint}). The previous statement of the $T1$ theorem can be found in \cite[Theorem 3.21]{tolsabook}. Although, there it is formulated for Euclidean spaces and measures with polynomial growth, it is also valid in spaces of homogeneous type. For details of this argument, see the honors thesis of Surath Fernando \cite{Fer} which extends the proof from \cite{t1tolsa} to spaces of homogeneous type.

Note that we can reduce the problem even further as it suffices to prove that there exists some constant $C:=C(\gamma, K, \G) \geq 1$ such that \begin{equation}
\label{T1}
\Vert T_{\mu,\varepsilon} \chi_Q \Vert_{L^2(\mu|_Q)}^2 \leq C \mu(Q).
\end{equation}
Indeed, recalling Definition~\ref{adjoint} and the discussion afterwards, $\tilde{K}$ is the kernel of the adjoint $\tilde{T}_{\mu,\varepsilon}$. Moreover, since $\tilde{K}(p)=K(p^{-1})$ it follows immediately that $\tilde{K}$ is a CZ kernel with the same constants $B, \beta$ as $K$. Moreover $\tilde{K}$ obeys the annular boundedness condition
since the the Hausdorff 1-measure is invariant on horizontal lines under the mapping $p \mapsto p^{-1}$, and the functions appearing in Definition~\ref{annular} are radial (and $d(p,0)=d(0,p^{-1})$).

We now perform a Littlewood-Paley decomposition of the operator $T$ as in \cite{CFOsios, ChoLi, ChoLiZimTra}.
Fix a smooth, even function $\psi:\mathbb{R} \to \mathbb{R}$
satisfying $\chi_{B(0,1/2)} \leq \psi \leq \chi_{B(0,2)}$ and define the $\G$-radial functions $\psi_j :\mathbb{G} \to \mathbb{R}$ as
$$
\psi_j(p) := \psi(2^jd(p,0)) 
\quad 
\text{for all } p \in \mathbb{G}, \, j \in \mathbb{Z}.
$$
Set $\eta_j := \psi_j - \psi_{j+1}$
and $K_{(j)} := \eta_j K$.
In particular, we have
$$
\supp K_{(j)} \subset B\left(0,2^{-(j-1)} \right) \setminus B \left(0,2^{-(j+2)} \right),
$$
and it is a standard exercise to verify that  $K_{j}$ is an $1$-dimensional CZ kernel with (growth and H\"older continuity) constants only depending on the corresponding constants of $K$.

Define the operator $T_{(j)}$ as
$$
T_{(j)}f(p) = \int_G K_{(j)}(q^{-1}p) f(q) \, d\mu (q)
$$
and the sum $S_N := \sum_{j \leq N} T_{(j)}$.
Note that, since $\text{diam}(\Gamma) < \infty$,
there is an index $N_0 \in \mathbb{Z}$
depending only on $\Gamma$ so that $G \subset B(p,2^{-(j+2)})$
for all $j < N_0$ and $p \in G$. Observe that properties \eqref{davuni} and  \eqref{davball} of the Christ cubes imply that we can choose $N_0$ such that $\Delta_j =\emptyset$ for $j<N_0$.
Since the map $q \mapsto K_{(j)}(q^{-1}p)$ vanishes on $B(p,2^{-(j+2)})$,
it follows that $T_{(j)}f \equiv 0$ when $j < N_0$.
Hence we may write $S_N = \sum_{N_0 \leq j \leq N} T_{(j)}$.

The following lemma justifies the above decomposition
by allowing us to approximate $T_{\mu, \varepsilon}$ by some $S_N$
for small values of $\varepsilon$. 
\begin{lemma}
\label{lm43}
Fix $N \in \mathbb{Z}$ and $2^{-N} \leq \varepsilon < 2^{-(N-1)}$. Then
$$
|S_N f(p) - T_{\mu, \varepsilon} f(p) | \lesssim M_{\mu}f(p)
\quad
\text{for all } f \in L_{\text{loc}}^1(\mu),
$$
and hence we have in particular
$$
\Vert T_{\mu, \varepsilon} \chi_R \Vert_{L^2(\mu|_R)}
\lesssim
\mu(R)^{1/2}
+
\Vert S_N \chi_R \Vert_{L^2(\mu|_R)}
\quad
\text{for all } R \in \Delta.
$$
\end{lemma}
The proof of Lemma \ref{lm43} is identical to the proof of  \cite[Lemma 3.3]{CFOsios}, and we omit it. As such, it remains to prove
\begin{equation}
\label{e-goal}
\Vert S_N \chi_R \Vert_{L^2(\mu|_R)}^2 \lesssim \mu(R)
\quad
\text{for all } R \in \Delta, \, N \in \mathbb{Z}.
\end{equation}


Fix $R \in \Delta$ and $N \in \mathbb{Z}$.
We will prove \eqref{e-goal} for these choices of indices.
Note that $R \in \Delta_J$ for some $J \geq N_0$.
We need only consider those terms in $S_N \chi_R$ for which $j \geq J-2$
(i.e. for the small supports). 
Indeed, for any $p \in R$, 
the mapping $q \mapsto K_{(j)}(q^{-1}p)$
vanishes on $B(p,2^{-(j+2)})$ by definition,
and note that $R \subset B(p,2^{-(j+2)})$ whenever $2^{-j} > 2^{-(J-2)}$ 
since $\diam(R) \leq 2^{-J}$ by property (\ref{davdiam}) in the definition of the Christ cubes.
Therefore, 
\begin{equation}
\label{J-2}
T_{(j)} \chi_R \equiv 0\mbox{ on }R\mbox{ whenever }j < J-2.
\end{equation}

Let us now decompose the $L^2$ norm of $S_N \chi_R$ into integrals over slices of the cube.
Define
$$
\partial_\rho R = \{ q \in R \, : \, \rho 2^{-(J+1)} < d(q,G \setminus R) \leq \rho 2^{-J} \}.
$$
According to condition (\ref{thinbdry}) for the Christ cubes,
$\mu(\partial_\rho R) \leq C_{\partial} \rho^{1/C_{\partial}} \mu (R)$.
Set $\rho(k) = 2^{1+J-k}$ so that, when $k$ is very large, $\partial_{\rho(k)}R$ is very thin.
Note that $\partial_{\rho(k)} R = \emptyset$ whenever $k < J$
since $\diam (R) \leq 2^{-J}$.
We can therefore write
$$
\Vert S_N \chi_R \Vert_{L^2(\mu|_R)}^2 
= \sum_{k \geq J} \int_{\partial_{\rho(k)} R} |S_N \chi_R|^2 \, d \mu.
$$
Why do we make this decomposition?
Fix a slice size $k \geq J$ and a point $p \in \partial_{\rho(k)}R$.
Then we have 
$d(p,G \setminus R) > 2^{-k}$, so $G \cap B(p,2^{-k}) \subset R$.
Also, for $j > k$,
the support of $q \mapsto K_{(j)}(q^{-1}p)$ is contained in $B\left(p,2^{-(j-1)} \right) \subset B\left(p,2^{-k} \right)$.
This allows us to write
\begin{equation}
\label{e-simp}
T_{(j)} \chi_R(p)
= \int_R K_{(j)}(q^{-1}p) \, d\mu (q)
= \int_{G \cap B(p,2^{-k})} K_{(j)}(q^{-1}p) \, d\mu (q)=\int_{G} K_{(j)}(q^{-1}p) \, d\mu (q)
\end{equation}
which will be essential when applying Proposition~\ref{areaform}.


We will also choose an index $J_0$ depending only on $\mathbb{G}$ and $\Gamma$.
The value of this index will be made clear later in the proof.

Writing $m(k) = \max \{ k,J_0 \}$, we have
\begin{align*}
\Vert S_N \chi_R \Vert_{L^2(\mu|_R)}^2  
&= \sum_{k \geq J} \int_{\partial_{\rho(k)} R} |S_N \chi_R|^2 \, d \mu \\
&\overset{\eqref{J-2}}{=}
\sum_{k \geq J} \int_{\partial_{\rho(k)} R} \left| \sum_{j=J-2}^N T_{(j)} \chi_R \right|^2 \, d \mu \\
&\lesssim
\sum_{k \geq J} \int_{\partial_{\rho(k)} R} \left| \sum_{j=J-2}^{m(k)} T_{(j)} \chi_R \right|^2 \, d \mu
+
\sum_{k \geq J} \int_{\partial_{\rho(k)} R} \left| \sum_{j=m(k)+1}^N T_{(j)} \chi_R \right|^2 \, d \mu \\
&:= 
\mathbf{A} + \mathbf{B}.
\end{align*}
(Here, we use the convention that $\sum_{n=a}^b x_n = 0$ when $a > b$.)
Thus, in order to prove \eqref{e-goal},
it suffices to bound $\mathbf{A}$ and $\mathbf{B}$ by a constant multiple of $\mu(R)$ where this constant depends only on $K, \G$, and $\gamma$.

Let us first establish a (very rough) bound for $\mathbf{A}$.
(This sum is over the larger annuli where \eqref{e-simp} may not hold.)
Fix $p \in \mathbb{G}$ and any $j \in \mathbb{Z}$.
Since the support of $q \mapsto K_{(j)}(q^{-1}p)$ is contained in
$B\left(p,2^{-(j-1)} \right)$,
it follows from \eqref{CZ1} that
$|K_{(j)}(q^{-1}p)| \lesssim 2^j$.
Thus, for any $f \in L^{\infty}(\mu)$, we have
$$
|T_{(j)} f(p)| 
\lesssim 
\Vert f \Vert_{\infty}
\int_{B\left(p,2^{-(j-1)} \right)} 2^j \, d\mu
=
\Vert f \Vert_{\infty} 2^j \mu\left( B\left(p,2^{-(j-1)} \right) \right)
\lesssim
\Vert f \Vert_{\infty}
$$
by the regularity \eqref{meas-reg} of $G$.
In particular, this gives for each $k \geq J$
$$
\sum_{j = J-2}^{m(k)} |T_{(j)}\chi_R(p)|
\lesssim 
m(k) - J + 3
\quad
\text{for any }
p \in \partial_{\rho(k)} R.
$$
Thus, by property (\ref{thinbdry}) of the Christ cube construction,
\begin{align*}
\mathbf{A} 
=
\sum_{k \geq J} \int_{\partial_{\rho(k)} R} \left| \sum_{j=J-2}^{m(k)} T_{(j)} \chi_R \right|^2 \, d \mu
&\lesssim 
\sum_{k \geq J} (m(k) - J + 3)^2 \mu(\partial_{\rho(k)} R) \\
&\lesssim
\sum_{k \geq J} (m(k) - J + 3)^2 \, 2^{(1+J-k)/C_{\partial}} \mu (R)
\lesssim
\mu (R),
\end{align*}
where for the last inequality we also used that $J \geq N_0$.

We will now bound $\mathbf{B}$.
Fix $k \geq J$ and $p \in  \partial_{\rho(k)} R$, 
and choose $t_0 \in [a,b]$ so that $\gamma(t_0) = p$.
Writing $\gamma(t) = (\gamma_1(t),\gamma_2(t)) \in \mathbb{R}^n \times \mathbb{R}^{N-n}$ for each $t \in [0,1]$, 
we define as before
$$
L(t) = p*((t-t_0)\gamma_1'(t_0),0) \quad \text{for all } t \in \mathbb{R}
$$ 
to be the horizontal approximation of $\gamma$ at $t_0$.

Using Proposition~\ref{areaform}, for any $m(k) < j \leq N$ we have
\begin{align*}
T_{(j)} \chi_R(p) 
= 
\int_R K_{(j)}(q^{-1}p) \, d \mu (q)
\overset{\eqref{e-simp}}{=}
\int_G K_{(j)}(q^{-1}p) \, d \mu (q)
=
\int_a^b K_{(j)}(\gamma(t)^{-1}p) | \gamma_1'(t) | \, d t,
\end{align*}
and we can write 
\begin{align}
\int_a^b K_{(j)}&(\gamma(t)^{-1}p) | \gamma_1'(t) | \, d t \nonumber \\
&= \int_a^b K_{(j)}(\gamma(t)^{-1}p) | \gamma_1'(t) | \, d t 
- \int_a^b K_{(j)}(L(t)^{-1}p) | \gamma_1'(t_0) | \, d t \label{e-ineq1} \\
&\hspace{2.3in}+ \int_a^b K_{(j)}(L(t)^{-1}p) | \gamma_1'(t_0) | \, d t \label{e-ineq2}.
\end{align}

We will first provide a bound on the sum over $j$
of \eqref{e-ineq2}. Let
$$
\tilde{L} := \{ L(s)^{-1} p \, : \, s \in [a,b] \} \subset \mathbb{R}^n \times \{ 0 \}.
$$
Proposition \ref{ubhlab} implies that $K$ satisfies the annular boundedness condition. Hence, using  Proposition~\ref{areaform} we get
\begin{equation}
\begin{split}
\label{sumb1j}
\Bigg| \sum_{j=m(k) + 1}^N \int_a^b&  K_{(j)}(L(t)^{-1}p) | \gamma_1'(t_0) | \, d t \Bigg| \\
&\leq 
\left| \sum_{j=m(k)+1}^N \int_{L([a,b])} K_{(j)}(q^{-1}p) \, d \mathcal{H}^1 (q) \right| \\
&=
\left| \int_{L([a,b])} \left(\psi(2^{m(k)+1} d(p,q) ) - \psi(2^{N+1} d(p,q) )\right) K(q^{-1}p) \, d \mathcal{H}^1 (q) \right|\\
&=
\left| \int_{\tilde{L}} \left(\psi(2^{m(k)+1} d(q,0) ) - \psi(2^{N+1} d(q,0))\right) K(q) \, d \mathcal{H}^1 (q) \right|\\
&\overset{\eqref{annular}}{ \lesssim} A.
\end{split}
\end{equation}

Let us now provide a bound for \eqref{e-ineq1}.
Fix an index $m(k) < j \leq N$.
Note that, for any $t \in \gamma^{-1}(B(p,2^{-(j-1)}) \cap G)$,
we have
\begin{equation}
\label{tiny}
|t-t_0|
\leq
\text{diam}(\gamma^{-1}(B(p,2^{-(j-1)}))) 
\overset{\eqref{smallballs}}{\lesssim} 2^{-j}.
\end{equation}
We have that
\begin{align}
\Bigg| \int_a^b K_{(j)}(\gamma(t)^{-1}&p) | \gamma_1'(t) | \, d t 
- \int_a^b K_{(j)}(L(t)^{-1}p) | \gamma_1'(t_0) | \, d t \Bigg| \nonumber \\
&\leq \int_a^b \left|K_{(j)}(\gamma(t)^{-1}p) - K_{(j)}(L(t)^{-1}p)\right| | \gamma_1'(t) | \, d t \label{1st} \\
&\hspace{1.5in} + \int_a^b \left| K_{(j)}(L(t)^{-1}p) \right| \left| | \gamma_1'(t_0) | - | \gamma_1'(t) | \right| \, d t. \label{2nd}
\end{align}

We will first bound \eqref{1st}.
We would like to apply the H\"{o}lder estimate \eqref{CZ2} directly
for each $t \in \gamma^{-1}(B(p,2^{-(j-1)}) \setminus B(p,2^{-(j+2)}) \cap G)$, 
but, in order to do so, we need $d(\gamma(t),L(t)) \leq d(\gamma(t),p)/2$.
If this bound does not hold, then 
we have by Proposition~\ref{flat}
$$
2^{-j} 
\lesssim d(\gamma(t),p)/2 
< d(\gamma(t),L(t))
\overset{\eqref{eqpromonti} \wedge \eqref{tiny}}{\lesssim} 2^{-j(1+ \frac{\alpha}{s})}.
$$
However, we may choose the index $J_0$ depending only on $\mathbb{G}$ and $\Gamma$
so that, for $j > m(k) \geq J_0$, this is a contradiction, and it must be true that $d(\gamma(t),L(t)) \leq d(\gamma(t),p)/2$. Hence,
$$
\left|K_{(j)}(\gamma(t)^{-1}p) - K_{(j)}(L(t)^{-1}p)\right|
\overset{\eqref{CZ2}}{\lesssim} \frac{d (\gamma(t),L(t))^{\beta}}{d( \gamma(t),p)^{1+ \beta}}
\lesssim \frac{|t-t_0|^{(1+\alpha/s)\beta}}{2^{-(j+2)(1+ \beta)}}
\overset{\eqref{eqpromonti} \wedge \eqref{tiny}}{\lesssim} 2^{j-\alpha \beta j/ s}
$$
for any $t \in \gamma^{-1}(B(p,2^{-(j-1)}) \setminus B(p,2^{-(j+2)}) \cap G)$.
Since $\supp K_{(j)} \subset B(p,2^{-(j-1)})$, 
\begin{equation}
\label{e-tech-bnd}
\int_a^b \left|K_{(j)}(\gamma(t)^{-1}p) - K_{(j)}(L(t)^{-1}p)\right| | \gamma_1'(t) | \, d t
\lesssim
2^{j-\alpha \beta j/ s} \int_{\gamma^{-1}(B(p,2^{-(j-1)}) \cap G)} | \gamma_1'(t) | \, d t 
\overset{\eqref{smallballs}}{\lesssim}
2^{-\alpha \beta j/ s}. 
\end{equation}

In order to bound \eqref{2nd}, we note again that 
$|K_{(j)}| \lesssim 2^j$
follows here from \eqref{CZ1}, 
using the H\"{o}lder continuity of $\gamma'$ 
we obtain the following bound on \eqref{2nd}:
\begin{equation}
\label{e-lame-bnd}
\int_a^b \left| K_{(j)}(L(t)^{-1}p) \right| \left| | \gamma_1'(t_0) | - | \gamma_1'(t) | \right| \, d t
\lesssim
2^j \int_{\gamma^{-1}(B(p,2^{-(j-1)}) \cap G)} |t-t_0|^{\alpha} \, d t
\overset{\eqref{smallballs} \wedge \eqref{tiny}}{\lesssim} 
2^{-\alpha j}.
\end{equation}

We are now ready to bound $\mathbf{B}$:
$$
\left| \sum_{j=m(k)+1}^N T_{(j)} \chi_R (p) \right|
\overset{\eqref{sumb1j} \wedge \eqref{e-tech-bnd} \wedge \eqref{e-lame-bnd}}{\lesssim} 
A + \sum_{j=k + 1}^N \left(2^{-\alpha \beta j / s} + 2^{-\alpha j} \right)
\lesssim
1.
$$
Hence by condition (\ref{thinbdry}) on the Christ cubes,
$$
\mathbf{B}
=
\sum_{k \geq J} \int_{\partial_{\rho(k)} R} \left| \sum_{j=m(k)+1}^N T_{(j)} \chi_R \right|^2 \, d \mu
\lesssim
\sum_{k \geq J} \mu (\partial_{\rho(k)} R)
\lesssim
\sum_{k \geq J}
2^{(1+J-k)/C_{\partial}} \mu (R)
\lesssim
\mu(R),
$$
where for the last inequality we also used that $J \geq N_0$. Hence, we have proved \eqref{T1} and thus we have established \eqref{t1G}.

Lemma \ref{sepsupports} together with \eqref{horizpartG} and \eqref{t1G} imply \eqref{l2bound}. Hence, we can apply Proposition \ref{goodl} and obtain for $p \in (1, \infty)$ and $f \in L^{p}(\nu)$
$$\Vert T_{\nu,*} f \Vert_{L^{p}(\nu)} \leq A_p \Vert f \Vert_{L^p(\nu)}.$$
Recalling the definition of $\nu$, this implies that
$$\Vert T_{\mathcal{H}^1|_{\Gamma},*} f \Vert_{L^{p}(\mathcal{H}^1|_{\Gamma})} \leq A_p \Vert f \Vert_{L^p(\mathcal{H}^1|_{\Gamma})}$$
for $p \in (1, \infty)$ and $f \in L^{p}(\mathcal{H}^1|_{\Gamma})$.
The proof is complete.
\end{proof}

\bibliographystyle{alpha}

\bibliography{SIObib}

\end{document}